\documentclass[a4paper,oneside,11pt] {amsart}
\reversemarginpar  % These 2 commands work very fine

\usepackage{eucal}
\usepackage{amsthm}
\usepackage{amsfonts}
\usepackage{amsmath}
\usepackage{amssymb}
\usepackage{mathrsfs}
\usepackage{epsfig}

 \usepackage{float}

\DeclareSymbolFont{SY}{U}{psy}{m}{n}
\DeclareMathSymbol{\emptyset}{\mathord}{SY}{'306}

\newcommand{\ncom}{\newcommand}
\ncom{\ul}{\underline}
\ncom{\ol}{\overline}
\ncom{\bq}{\begin{equation}}
\ncom{\eq}{\end{equation}}
\ncom{\beqn}{\begin{eqnarray*}}
\ncom{\eeqn}{\end{eqnarray*}}
\ncom{\beq}{\begin{eqnarray}}
\ncom{\eeq}{\end{eqnarray}}
\ncom{\nno}{\nonumber}
\ncom{\rar}{\rightarrow}
\ncom{\Rar}{\Rightarrow}
\ncom{\noin}{\noindent}
\ncom{\bc}{\begin{centre}}
\ncom{\ec}{\end{centre}}
\ncom{\sz}{\scriptsize}
\ncom{\rf}{\ref}
\ncom{\sgm}{\sigma}
\ncom{\Sgm}{\Sigma}
\ncom{\dt}{\delta}
\ncom{\Dt}{Delta}

\ncom{\lmd}{\lambda}
\ncom{\Lmd}{\Lambda}
\ncom{\eps}{\epsilon}
\ncom{\pcc}{\stackrel{P}{>}}
\ncom{\dist}{{\rm\,dist}}
\ncom{\sspan}{{\rm\,span}}
\ncom{\re}{{\rm Re\,}}
\ncom{\im}{{\rm Im\,}}
\ncom{\sgn}{{\rm sgn\,}}
\ncom{\ba}{\begin{array}}
\ncom{\ea}{\end{array}}
\ncom{\eop}{\hfill{{\rule{2.5mm}{2.5mm}}}}
\ncom{\eoe}{\hfill{{\rule{1.5mm}{1.5mm}}}}
\ncom{\eof}{\hfill{{\rule{1.5mm}{1.5mm}}}}
\ncom{\hone}{\mbox{\hspace{1em}}}
\ncom{\htwo}{\mbox{\hspace{2em}}}
\ncom{\hthree}{\mbox{\hspace{3em}}}
\ncom{\hfour}{\mbox{\hspace{4em}}}
\ncom{\hsev}{\mbox{\hspace{7em}}}
\ncom{\vone}{\vskip 2ex}
\ncom{\cH}{{\mathcal H}}
\ncom{\vtwo}{\vskip 4ex}
\ncom{\vonee}{\vskip 1.5ex}
\ncom{\vthree}{\vskip 6ex}
\ncom{\vfour}{\vspace*{8ex}}
\ncom{\norm}{\|\;\;\|}
\ncom{\integ}[4]{\int_{#1}^{#2}\,{#3}\,d{#4}}
\ncom{\inp}[2]{\langle{#1},\,{#2} \rangle}
\ncom{\Inp}[2]{\big\langle{#1},\,{#2} \big\rangle}
\ncom{\vspan}[1]{{{\rm\,span}\#1 \}}}
\ncom{\dm}[1]{\displaystyle {#1}}
\ncom{\Hom}{\operatorname{Hom}}
\ncom{\Hol}{\operatorname{Hol}}

\ncom{\defin} {\overset {\text {\rm def} }{=}}

\newtheorem{theorem}{\bf Theorem}[section]
\newtheorem{example}[theorem]{\bf Example}%[section]
\newtheorem{proposition}[theorem]{\bf Proposition}%[section]
\newtheorem{corollary}[theorem]{\bf Corollary}%[section]
\newtheorem{lemma}[theorem]{\bf Lemma}%[section]

\newtheorem{question}[theorem]{\bf Question}
\newtheorem{conjecture}[theorem]{\bf Conjecture}

\newtheoremstyle
    {remarkstyle}
    {}
    {11pt}
    {}
    {}
    {\bfseries}
    {:}
    {     }
    {\thmname{#1} \thmnumber{#2} }

\theoremstyle{remarkstyle}

\newtheorem{remark}[theorem]{\bf Remark}%[section]
\newtheorem{definition}[theorem]{\bf Definition}%[section]
\ncom{\bbeta}{\beta}
\renewcommand{\epsilon}{\varepsilon}
\renewcommand{\kappa}{\varkappa}

\begin{document}

 \baselineskip=16pt

\title [On sum of two subnormal kernels] {On sum of two subnormal kernels}

\author[S. Ghara]{Soumitra Ghara}
\address[S. Ghara]{Department of Mathematics, Indian Institute of Science,
Bangalore 560012, India} \email{ghara90@gmail.com}

\author[S. Kumar]{Surjit Kumar}
\address[S. Kumar]{Department of Mathematics, Indian Institute of Science,
Bangalore 560012, India} \email{surjit.iitk@gmail.com}

 \thanks{Work of the first author was supported by CSIR SPM Fellowship and
 work of the second author was supported by Inspire Faculty Fellowship.}

\date{}

\begin{abstract}
We show, by means of a class of examples, that if $K_1$ and $K_2$ are
two positive definite kernels on the unit disc such that the multiplication
by the coordinate function on the corresponding reproducing kernel Hilbert space
is subnormal, then the multiplication operator on the Hilbert space determined by
their sum $K_1+K_2$ need not be subnormal. This settles a recent conjecture of
Gregory T. Adams, Nathan S. Feldman and Paul J. McGuire in the negative.
We also discuss some cases for which the answer is affirmative.
\end{abstract}

 \renewcommand{\thefootnote}{}

\maketitle

%\setcounter{tocdepth}{2} \tableofcontents

%\newpage

%%%%%%%%%%%%%%%%%%%%%%%%%%%%%%%%%%%%%%%%%%%%%%%%%%%%%%%%%%%%%%%%%%%%%%%%%%%%%%%%%%%%%%%%%%%
%%%%%%%%%%%%%%%%%%%%%%%%%%%%%%%%%%%%%%%%%%%%%%%%%%%%%%%%%%%%%%%%%%%%%%%%%%%%%%%%%%%%%%%%%%%

\section{Introduction}
 Let ${\mathcal H}$ be a complex separable
Hilbert space.
%with the inner product $\langle x,y \rangle_{\mathcal
%H} \ (x,y \in {\mathcal H})$ and the corresponding norm
%$||x||_{\mathcal H} \ (x \in {\mathcal H})$.
Let ${B}({\mathcal H})$
denote the Banach algebra of bounded linear operators on $\mathcal H.$
Recall that an operator $T$ in ${B}({\mathcal H})$ is said to be {\it subnormal}
if there exists a Hilbert space $\mathcal K\supset\mathcal{H}$ and a normal operator $N$ in ${B}({\mathcal K})$ such that
$N(\mathcal H)\subset \mathcal H$ and $N|_{\mathcal{H}}=T.$ For the basic theory of subnormal operators, we refer to \cite{Co}.

{\it Completely hyperexpansive} operators were introduced in \cite{At-2}.
An operator $T\in B(\mathcal H)$ is said to be completely hyperexpansive if
\beqn \sum_{j=0}^n(-1)^j \binom{n}{j}T^{*j}T^j\leq 0 \qquad (n\geq 1).\eeqn

The theory of subnormal and completely hyperexpansive operators are closely
related with the theory {\it completely monotone} and {\it completely alternating} sequences
(cf.  \cite{At-1}, \cite{At-2} ).

Let $\mathbb Z_+$ denote the set of non-negative integers.
A sequence $\{a_k\}_{k \in \mathbb Z_+}$ of positive real numbers is said to be a completely monotone if

\beqn \sum_{j=0}^n (-1)^j{n\choose j}a_{m+j}\geq 0 \qquad (m, n\geq 0).\eeqn
It is well-known that a sequence $\{a_k\}_{k \in \mathbb Z_+}$ is completely monotone
if and only if the sequence $\{a_k\}_{k \in \mathbb Z_+}$ is a Hausdorff moment sequence, that is, there
exists a positive measure $\nu$ supported in $[0,1]$ such that
 $a_k=\int_{[0,1]}x^k d\nu(x)$ for all $k \in \mathbb Z_+$ (cf. \cite{B-C-R}).

Similarly, a sequence $\{a_k\}_{k \in \mathbb Z_+}$ of positive real numbers is said to be  completely alternating if
\beqn \sum_{j=0}^n (-1)^j{n\choose j}a_{m+j}\leq 0 \qquad (m\geq 0,n\geq 1).\eeqn
Note that $\{a_k\}_{k \in \mathbb Z_+}$ is completely alternating if and only if
the sequence $\{\Delta a_k\}_{k \in \mathbb Z_+}$ is completely monotone, where $\Delta a_k:=a_{k+1}-a_k.$

Let $\mathcal H(K)$ be a reproducing kernel Hilbert space consisting of holomorphic functions on the open
unit disc $\mathbb D:=\{z\in \mathbb C: |z|<1\}$ with reproducing
kernel $K(z,w)=\sum_{k \in \mathbb Z_+} a_k(z\bar{w})^k.$ Thus $K$ is holomorphic
in the first variable and anti-holomorphic in the second. We make the assumption
that a kernel function is always holomorphic in the first variable
and anti-holomorphic in the second throughout this note. Consider the
operator $M_z$ of multiplication by the coordinate function $z$ on
$\mathcal H(K).$ 
As is well-known,
such a multiplication operator is unitarily equivalent to a weighted shift operator $W$ on the
sequence space
$$\ell^2(\mathbf a):= \{\mathbf x:=(x_0,x_1, \ldots ): \|\mathbf x\|^2 = \sum_{k=0}^\infty \frac{|x_k|^2}{a_k} < \infty \}$$
with $W e_n = \sqrt{\tfrac{a_n}{a_{n+1}}} e_{n+1},$ where $e_n$ is the standard unit vector.
Assume that $M_z$ is bounded. Then $M_z$ is a
contractive subnormal if and only if the sequence
$\{\frac{1}{a_k}\}_{k \in \mathbb Z_+}$ is a Hausdorff moment sequence.
On the other hand, $M_z$ on $\mathcal{H}(K)$ is completely hyperexpansive
if and only if the sequence $\{\frac{1}{a_k}\}_{k \in \mathbb Z_+}$ is completely alternating
(cf. \cite[Proposition 3]{At-2}). 

For any two positive definite kernels $K_1$ and $K_2,$ their sum $K_1+K_2$
is again a positive definite kernel and therefore determines a Hilbert
space $\mathcal H(K_1+K_2)$ of functions. It was shown in \cite{Aro} that
$$\mathcal H(K_1+K_2) = \{f=f_1+f_2:
f_1 \in \mathcal H(K_1), f_2 \in \mathcal H(K_2)\},$$
is a Hilbert space with the norm given by
\beqn \|f\|^2_{\mathcal H(K_1+K_2)} := \inf \left\{\|f_1\|^2_{\mathcal H(K_1)}+\|f_2\|^2_{\mathcal H(K_2)}
: f=f_1+f_2,~f_1 \in \mathcal H(K_1),~f_2 \in \mathcal H(K_2)\right\}.\eeqn
Sum of two kernel functions is also discussed by Salinas in \cite{Sal}. He proved that if
$K_1$ and $K_2$ are generalized Bergman kernels (for definition, refer to \cite{Cu-Sa-1} ), then so is $K_1+K_2.$
Although not explicitly stated in \cite{Aro}, it is not hard to verify that the multiplication operator
$M_z$ on $\mathcal H(K_1+K_2)$ is unitarily equivalent to
the operator $P_{\mathcal M^{\perp}}(M_{z,1}\oplus M_{z,2})|_{\mathcal M^{\perp}},$ where $M_{z,i}$ is the 
operator of multiplication by the coordinate function $z$ on $\mathcal H(K_i)$ and
$$\mathcal M = \{(g, -g) \in \mathcal H(K_1)
\oplus \mathcal H(K_2) : g \in \mathcal H(K_1)\cap \mathcal H(K_2)\} \subseteq \mathcal H(K_1) \oplus \mathcal H(K_2).$$
Evidently, if $M_{z,1}$ and $M_{z,2}$ are subnormal, then so is $M_{z,1}\oplus M_{z,2}.$ Here, we discuss the subnormality
of the compression $P_{\mathcal M^{\perp}}(M_{z,1}\oplus M_{z,2})|_{\mathcal M^{\perp}}$ for a class of kernels.
In particular, we show that the subnormality of $M_{z,1}$ and $M_{z,2}$ need not imply
$P_{\mathcal M^{\perp}}(M_{z,1}\oplus M_{z,2})|_{\mathcal M^{\perp}}$
is subnormal.

A similar question on subnormality involving the point-wise product of two positive definite kernels was 
raised in \cite{Sal}.
Recall that the product $K_1 K_2$ of two positive definite kernels defined on, say the unit disc $\mathbb D,$   is also a positive
definite kernel on $\mathbb D.$
Indeed, if $\mathcal H(K_1) \otimes \mathcal H(K_2) \subseteq \mbox{\rm Hol}(\mathbb D\times \mathbb D)$ be the usual tensor product of
the two Hilbert spaces $\mathcal H(K_1)$ and $\mathcal H(K_2),$ and $\mathcal{N}$ is the subspace
$\{h\in \mathcal H(K_1)\otimes \mathcal H(K_2):h(z,z)=0, z \in \mathbb D\},$ then the operator $M_{z,1}\otimes I$ acting
on the Hilbert space $\mathcal H(K_1) \otimes \mathcal H(K_2)$  compressed to $\mathcal N^\perp$ is unitarily
equivalent to the multiplication operator $M_z$ on the Hilbert space $\mathcal H(K_1K_2).$
If $M_{z,1}$ is subnormal on $\mathcal H(K_1),$ then so is the operator $M_{z,1}\otimes I.$

The answer to the question of subnormality, both in the case of the sum as
well as the product, is affirmative in several examples. 
For over thirty years, the question of whether the compression to $\mathcal N^\perp$ is subnormal
had remained open. Recently, a counter-example has been found,
see \cite[Theorem 1.5]{A-Ch-1}. The conjecture below is similar except that it involves the sum of two kernels.

\begin{conjecture}\mbox{\rm \bf (}\cite[pp. 22]{A-F-M} \!\!\mbox{\rm \bf ).}\label{conj}
 Let $K_1(z,w)=\sum_{k \in \mathbb Z_+} a_k (z\bar{w})^k$ and $K_2(z,w)=\sum_{k \in \mathbb Z_+} b_k(z\bar{w})^k$
 be any two reproducing kernels satisfying:

\begin{itemize}
\item[(a)] $\lim \frac{a_k}{a_{k+1}}=\lim \frac{b_k}{b_{k+1}}=1$
\item[(b)]$\lim a_k=\lim b_k=\infty$
\item[(c)]$\frac{1}{a_k}=\int_{[0, 1]} t^k d\nu_1(t)$ and $\frac{1}{b_k}=\int_{[0, 1]} t^k d\nu_2(t)$ for all $k \in \mathbb Z_+.$
\end{itemize}
Then the multiplication operator $M_z$ on $\mathcal
H(K_1+K_2)$ is a subnormal operator.
\end{conjecture}

An equivalent formulation, in terms of the moment sequence criterion, of the conjecture is the following.
If $\{\tfrac{1}{a_k}\}_{k \in \mathbb Z_+}$ and $\{\tfrac{1}{b_k}\}_{k \in \mathbb Z_+}$ are Hausdorff moment sequences,
does it necessarily follow that $\{\tfrac{1}{a_k+b_k}\}_{k \in \mathbb Z_+}$ is also a Hausdorff moment sequence?

Like the case of the product of two kernels, here we give a class of counter examples to the 
conjecture stated above. Paul McGuire, in a private communication, has   informed the authors 
of an example that they had found. In fact, he says that their example is one of the examples discussed in this note.

%The case of the product
%of two kernels was settled in \cite[Theorem 1.5]{A-Ch-1} and we settle the case of the sum.

These two cases suggest that it may be fruitful to ask when the compression of a subnormal operator to an invariant subspace is again subnormal.

%In a recent paper , the following conjecture is made.

The paper is organized as follows. In section $2$, we provide a class of counter-examples 
which settles the Conjecture \ref{conj}. We also discuss some cases for which answer to this conjecture is affirmative.
In the last section, we try to answer analogously in a certain class of weighted multi-shifts.

\section{Sum of two subnormal reproducing kernels need not be subnormal}

For the construction of counter-examples to the conjecture, 
we make use of the following result, borrowed from \cite[Proposition 4.3]{A-Ch}.
\begin{proposition}\label{nms}
For distinct positive real numbers $a_0,...,a_n$ and non-zero real
numbers $b_0,...,b_n,$ consider the polynomial $p(x)=\Pi_{k=0}^n
(x+a_k+ib_k)(x+a_k-ib_k).$ Then the sequence
$\{\frac{1}{p(l)}\}_{l\in \mathbb Z_+}$ is never a Hausdorff moment
sequence.
\end{proposition}

For $r>0,$ let $K_r$ be a positive definite kernel given by
\beqn K_{r}(z,w) :=\sum_{k \in \mathbb Z_+} \frac{k+r}{r} (z\bar{w})^k \qquad (z,w\in
\mathbb D).\eeqn
 The case $r=1$ corresponds to the Bergman kernel.
 It is easy to see that the multiplication operator $M_z$ on $\mathcal H(K_{r})$
 is a contractive and subnormal and the representing measure is $r x^{r-1} dx.$

For $s,t>0,$ consider the multiplication operator $M_z$ on
$\mathcal H(K_{s, t}),$ where
\beqn K_{s,t}(z,w):=\sum_{k \in \mathbb Z_+} \frac{(k+s)(k+t)}{st}(z\bar{w})^k \qquad (z,w\in \mathbb D).\eeqn
The case $s =1$ and $t =2,$ corresponds to the kernel
$(1- z \bar{w})^{-3}.$
Note that $M_z$ is a contractive subnormal with the representing
measure $\nu$ is given by

\begin{eqnarray}  d\nu(x) &=& \left \{\begin{array} {ll}
     -s^2x^{s-1}\log{x}~ dx \; \;\mbox{if}~s=t
      \cr st\frac{x^{t-1}-x^{s-1}}{s-t} dx~ \; \; \;\mbox{if}~ s \neq t.
        \end{array} \right. \nonumber \end{eqnarray}

One  easily verifies that $K_r$ and $K_{s,t}$ both
satisfy all the conditions $(a), (b)$ and $(c)$ of the Conjecture \ref{conj}.
But the multiplication operator on their sum need not be subnormal for all possible choices of $s,t>0$. This
follows from the following theorem.

\begin{theorem} \label{thm}
The multiplication operator $M_z$ on $\mathcal H(K_r+K_{s,t})$ is subnormal if and only if
\beq \label{eq0.1} (rs+st+tr)^2\geq 8 r^2st.\eeq
\end{theorem}

\begin{proof}
Notice that \beqn (K_r+K_{s,t})(z,w)=\sum_{k \in \mathbb Z_+}
\frac{k^2+(s+t+\frac{st}{r})k+2st} {st}(z\bar{w})^k \qquad (z,w\in \mathbb D).\eeqn

The roots of the polynomial $x^2+(s+t+\frac{st}{r})x+2st$
are
\beqn x_1 &:=& \frac{-(s+t+\frac{st}{r})+\sqrt{(s+t+\frac{st}{r})^2-8st}}{2}
~\mbox{and}~ \cr x_2 &:=& \frac{-(s+t+\frac{st}{r})-\sqrt{(s+t+\frac{st}{r})^2-8st}}{2}.\eeqn
Suppose that $ (rs+st+tr)^2\geq 8 r^2st.$ Then the kernel
$K_r+K_{s,t}$ will be of the form $2K_{s^{\prime}, t^{\prime}}$ where $s^{\prime}= -x_1$ and $t^{\prime}=-x_2.$
Hence, $M_z$ on $\mathcal  H(K_r+K_{s,t})$ is a subnormal operator.

Conversely, assume that $ (rs+st+tr)^2 < 8 r^2st.$
By Proposition $\ref{nms}$, it follows that $M_z$ on $\mathcal  H(K_r+K_{s,t})$
can not be subnormal.
\end{proof}

\begin{remark}
 If we choose $s=1, ~t=2$ and $r > 2,$ then the inequality \eqref{eq0.1} is not valid.
\end{remark}

We also point out that if $K_1$ and $K_2$ are any two reproducing kernels such that
the multiplication operators on $\mathcal{H}(K_1)$
and $\mathcal{H}(K_2)$ are hyponormal, then the multiplication operator on $\mathcal{H}(K_1+K_2)$
need not be hyponormal. An example illustrating this is given below.

\begin{example}
For any $s,t>0,$ consider the reproducing kernel $K^{s,t}$ given by
\beqn K^{s,t}(z,w):=1+sz\bar{w}+s^2(z\bar{w})^2+t\frac{(z\bar{w})^3}{1-z\bar{w}}.\eeqn
Note that $K^{s,t}$ defines a reproducing kernel on the unit disc $\mathbb D$
and the multiplication operator $M_z$ on $\mathcal{H}(K^{s,t})$ can be seen as
a weighted shift operator with weights $(\sqrt{\frac {1}{s}},\sqrt{\frac {1}{s}},\sqrt{\frac{s^2}{t}},1,1, \cdots).$
Thus, it follows that $M_z$ on $\mathcal{H}(K^{s,t})$ is hyponormal if and only if
$s^2\leq t\leq s^3.$

Observe that $M_z$ on $\mathcal{H}(K^{s,t}+K^{s',t'})$ can be realized as a weighted shift operator with weights
$(\sqrt{\frac{2}{s+s'}},\sqrt{\frac{s+s'}{s^2+s'^2}},\sqrt{\frac{s^2+s'^2}{t+t'}},1,1,\cdots).$ For the hyponormality of this weighted shift operator,
it is necessary that $\frac{2}{s+s'}\leq \frac{s+s'}{s^2+s'^2},$ which is true only when $s=s'.$

\end{example}

We remark that this is different from the case of the product of two kernels, where,
the hyponormality of the multiplication operator on the Hilbert space $\mathcal H(K_1 K_2)$
follows as soon as we assume they are hyponormal on the two Hilbert spaces $\mathcal H(K_1)$ and $\mathcal H(K_2)$, see \cite{B-S}.

%%%%%%%%%%%%%%%%%%%%%%%%%%%%%%%%%%%%%%%%%%%%%%%%%%%%%%%%%%%%%%%%%%%%%%%%%%%%%%%%%%%%%%%%%%%%%%%%%%%%%%%%%%%%%%%%%%%%%%%%%%%%%%%%%%%%%%%%%%%
%%%%%%%%%%%%%%%%%%%%%%%%%%%%%%%%%%%%%%%%%%%%%%%%%%%%%%%%%%%%%%%%%%%%%%%%%%%%%%%%%%%%%%%%%%%%%%%%%%%%%%%%%%%%%%%%%%%%%%%%%%%%%%%%%%%%%%%%%%%

If $T \in {B}({\mathcal H})$ is left invertible then the operator $T^{\prime}$ given by $T^{\prime}=T(T^*T)^{-1}$
is said to be the operator {\it Cauchy dual} to $T.$
The following result has been already recorded in \cite[Proposition 6]{At-2}, which may be paraphrased as follows:
\begin{theorem} \label{At-thm}
Let $K(z,w)=\sum_{k \in \mathbb Z_+} a_k(z\bar{w})^k$ be a positive definite kernel on $\mathbb D$
and $M_z$ be the multiplication operator on $\mathcal H(K).$ Assume that $M_z$ is left invertible.
Then the followings are equivalent:
\begin{itemize}
 \item[(i)] $\{a_k\}_{k \in \mathbb Z_+}$ is a completely alternating sequence.
 \item[(ii)] The Cauchy dual $M^{\prime}_z$ of $M_z$ is completely hyperexpansive.
 \item[(iii)] For all $t>0$, $\{\frac{1}{t(a_k-1)+1}\}_{k \in \mathbb Z_+}$ is a completely monotone sequence.
 \item[(iv)]For all $t>0$, the multiplication operator $M_z$ on $\mathcal{H}(tK+(1-t)S)$ is contractive subnormal, 
 where $S$ is the Szeg$\ddot{o}$ kernel on $\mathbb D$.
\end{itemize}

\end{theorem}

\begin{remark} \label{rem2.6}
If $\{a_k\}_{k \in \mathbb Z_+}$ is a completely alternating sequence then by putting $t=1$ in part (iii) of Theorem \ref{At-thm},
it follows that $\{\frac{1}{a_k}\}_{k \in \mathbb Z_+}$ is a completely monotone sequence.
\end{remark}

\begin{corollary}\label{cor2.6}
Let $K_1(z,w)=\sum_{k \in \mathbb Z_+} a_k (z\bar{w})^k$ and $K_2(z,w)=\sum_{k \in \mathbb Z_+} b_k(z\bar{w})^k$  be any two reproducing kernels
such that $\{a_k\}_{k \in \mathbb Z_+}$ and $\{b_k\}_{k \in \mathbb Z_+}$ are completely alternating sequences, then the multiplication operator
$M_z$ on $\mathcal{H}(K_1+K_2)$ is subnormal.
\end{corollary}
\begin{proof}
 It is easy to verify that the sum of two completely alternating sequences is completely alternating. The desired
 conclusion follows immediately from Remark \ref{rem2.6}.

\end{proof}

\begin{remark}
Note that $\{\frac{k+r}{r}\}_{k \in \mathbb Z_+}$ is a completely alternating sequence but the sequence
$\{\frac{(k+s)(k+t)}{st}\}_{k \in \mathbb Z_+}$ is not completely alternating. So, the reproducing kernels
$K_r$ and $K_{s,t}$ discussed in Theorem \ref{thm} does not satisfy the hypothesis of Corollary \ref{cor2.6}.

\end{remark}

\begin{proposition} \label{pro2.9}
Let $K(z,w)=\sum_{k \in \mathbb Z_+} a_k (z\bar{w})^k$
be any positive definite kernel such that the multiplication operator $M_z$ 
on $\mathcal{H}(S+K)$ is subnormal. Then the multiplication operator on 
$\mathcal H(K)$ is subnormal.
\end{proposition}

\begin{proof}
From subnormality of $M_z$ on $\mathcal{H}(S+K),$ it follows that 
$\{\tfrac{1}{1+a_k}\}_{k \in \mathbb Z_+}$ is a completely monotone sequence. 
Thus, $\{1-\tfrac{1}{1+a_k}\}_{k \in \mathbb Z_+}$ is completely alternating. Note that
\beqn (a_k)^{-1} &=&(1+a_k)^{-1}(1-\tfrac{1}{1+a_k})^{-1}
\cr &=& \sum_{j=1}^{\infty}\left(\tfrac{1}{1+a_k}\right)^j.\eeqn

Observe that $\{\tfrac{1}{(1+a_k)^j}\}_{k \in \mathbb Z_+}$ is a completely monotone sequence
for all $j \geq 1.$ Now, being the limit of completely monotone sequences, $\{a_k^{-1}\}_{k \in \mathbb Z_+}$
is completely monotone.

\end{proof}

\begin{remark}
We have following remarks:
\begin{itemize}
 \item [(i)] The converse of the Proposition \ref{pro2.9} is not true (see the example discussed in part (ii) of the Remark \ref{rem2.13}).
 \item [(ii)] If we replace Szeg$\ddot{o}$ kernel $S$ by Bergman kernel then the conclusion of the Proposition \ref{pro2.9}
 need not be true. For example, by using Proposition \ref{nms}, one may choose $\alpha > 0$ such that 
 the sequence $\{\frac{1}{k^2+\alpha k + 1}\}_{k \in \mathbb Z_+}$ is not completely monotone
 but the sequence $\{\frac{1}{k^2+(\alpha +1)k+2}\}_{k \in \mathbb Z_+}$ is completely monotone.
\end{itemize}

\end{remark}

We use the convenient Pochhammer symbol given by $(x)_n:=\frac{\Gamma(x+n)}{\Gamma(x)},$
where $\Gamma$ denotes the gamma function.

For $\lambda,\mu>0,$ consider the positive definite kernel \beqn
K_{\lambda,\mu}(z,w)=\sum_{k \in \mathbb Z_+}
\frac{(\lambda)_k}{(\mu)_k}(z\bar{w})^k \qquad (z,w\in \mathbb
D).\eeqn It is easy to see that the case $\mu=1$ corresponds to the
kernel $(1-z\bar{w})^{-\lambda}.$ Note that the multiplication
operator $M_z$ on $\mathcal H( K_{\lambda,\mu})$ may be realized as
a weighted shift operator with weight sequence
$\{\sqrt{\tfrac{k+\mu}{k+\lambda}}\}_{k \in \mathbb Z_+}.$

First part of the following theorem is proved in \cite{C-P-Y} and the representing measure is given in \cite[Lemma 2.2]{C-D}.
Here, we provide a proof for the second part only.

\begin{theorem} \label{thm2.8}
The multiplication operator $M_z$ on $\mathcal{H}(K_{\lambda,\mu})$ is
\begin{itemize}
\item[(i)]subnormal if and only if $\lambda\geq \mu.$ In this case, the representing measure $\nu$ of $M_z$ is given by

\begin{eqnarray}  d\nu(x) &=& \left \{\begin{array} {ll}
     \frac{\Gamma(\lambda)}{\Gamma(\mu)\Gamma(\lambda-\mu)} x^{\mu-1}(1-x)^{\lambda-\mu-1} dx \; \;\mbox{if}~ \lambda > \mu
      \cr \delta_1(x)dx \hskip4.0cm  \mbox{if}~\lambda = \mu,
        \end{array} \right. \nonumber \end{eqnarray}
where $\delta_1$ is the Dirac delta function.
% \beqn \frac{\Gamma(\lambda)}{\Gamma(\mu)\Gamma(\lambda-\mu)} x^{\mu-1}(1-x)^{\lambda-\mu-1} dx.\eeqn

\item[(ii)]completely hyperexpansive if and only if $\lambda\leq \mu\leq \lambda+1.$
\end{itemize}

\end{theorem}

\begin{proof}
The multiplication operator $M_z$ on $\mathcal H(K_{\lambda,\mu})$
is completely hyperexpansive if and only if the sequence $\{\frac{(\mu)_k}{(\lambda)_k}\}_{k \in \mathbb Z_+}$
is completely alternating.
% Recall that a sequence $\{a_k\}_{k \in \mathbb Z_+}$ is completely alternating
% if and only if the sequence $\{a_{k+1}-a_k\}$ is completely monotone.
Here
\begin{align*}
\Delta \left(\frac{(\mu)_k}{(\lambda)_k}\right) =\frac{(\mu)_{k+1}}{(\lambda)_{k+1}}-\frac{(\mu)_k}{(\lambda)_k}
&=\frac {(\mu)_{k+1}-(\mu)_k(\lambda+k-1)}{(\lambda)_{k+1}}\\
&=\frac{\mu-\lambda}{\lambda}\frac{(\mu)_k}{(\lambda+1)_k}.
\end{align*}
By first part of this theorem, $\{\frac{\mu-\lambda}{\lambda}\frac{(\mu)_k}{(\lambda+1)_k}\}_{k \in \mathbb Z_+}$
is a completely monotone sequence if and only if $\lambda\leq \mu\leq \lambda+1.$ This completes the proof.

\end{proof}

The following proposition gives a sufficient condition for the subnormality of multiplication operator on Hilbert space 
determined by sum of two kernels belonging to the class  $K_{\lambda, \mu}.$

\begin{proposition}\label{prop2.9}
Let $0<\mu\leq \lambda'\leq \lambda\leq \lambda'+1.$ Then the multiplication operator
$M_z$ on $\mathcal{H}(K_{\lambda,\mu}+K_{\lambda',\mu})$ is contractive subnormal.

\end{proposition}

\begin{proof}
Observe that
\beqn (K_{\lambda,\mu}+K_{\lambda',\mu})(z,w)=\sum_{k \in \mathbb Z_+} \frac{(\lambda)_k+(\lambda')_k}{(\mu)_k}(z\bar{w})^k
\qquad (z,w\in \mathbb D) \eeqn
and \beqn \frac{(\mu)_k}{(\lambda)_k+(\lambda')_k}=
\frac{(\mu)_k}{(\lambda')_k}\frac{1}{1+\frac{(\lambda)_k}{(\lambda')_k}}.\eeqn
Since $\mu\leq \lambda',$ it follows from part (i) of Theorem \ref{thm2.8}
that $\{\frac{(\mu)_k}{(\lambda')_k}\}_{k \in \mathbb Z_+}$ is completely monotone.

If $\lambda'\leq \lambda\leq \lambda'+1,$ then by
part (ii) of Theorem \ref{thm2.8}, the sequence $\{\frac{(\lambda)_k}{(\lambda')_k}\}_{k \in \mathbb Z_+}$
is a completely alternating sequence. So is the sequence $\{1+\frac{(\lambda)_k}{(\lambda')_k}\}_{k \in \mathbb Z_+}.$ Hence, by Remark \ref{rem2.6},
$\left\{\frac{1}{1+\frac{(\lambda)_k}{(\lambda')_k}}\right\}_{k \in \mathbb Z_+}$ is a completely monotone sequence.
Thus, being a product of two completely monotone sequences, $\{\frac{(\mu)_k}{(\lambda)_k+(\lambda')_k}\}_{k \in \mathbb Z_+}$ is a completely monotone sequence.
This completes the proof.
\end{proof}

\begin{remark} \label{rem2.13}
Here are some remarks:
\begin{itemize}
\item[(i)]The case when $\mu<\lambda'$ and $\lambda=\lambda'+1.$ The representing measure for the
sequence $\left \{\frac{1}{1+\frac{(\lambda)_k}{(\lambda')_k}}\right\}_{k \in \mathbb Z_+}$ is $\lambda'x^{2\lambda'-1}dx.$
The representing measure for the
sequence $\{\frac{(\mu)_k}{(\lambda')_k}\}_{k \in \mathbb Z_+}$ is given in part (i) of Theorem \ref{thm2.8}. Thus, using
Remark $2.4$ of \cite{A-Ch}, one may obtain the representing measure for $M_z$ on
$\mathcal{H}(K_{\lambda,\mu}+K_{\lambda',\mu})$ to be
given by
$$d\nu(x)=\frac {\lambda'\Gamma(\lambda')}{\Gamma(\mu)\Gamma(\lambda'-\mu)}x^{2\lambda'-1}
\left(\int_0^{1-x}t^{\lambda'-\mu-1}(1-t)^{\mu-2\lambda'-1}dt\right)dx.$$
But in general, when $\lambda<\lambda'+1,$ we do not know the representing measure for
the sequence $\left \{\frac{1}{1+\frac{(\lambda)_k}{(\lambda')_k}}\right\}_{k \in \mathbb Z_+}$
as well as for the sequence $\left \{\frac{(\mu)_k}{(\lambda)_k+(\lambda')_k}\right\}_{k\in \mathbb Z_+.}$

\item[(ii)]
Consider the kernel  $(K_{1,1}+K_{3,1})(z,w)=\sum_{k \in \mathbb Z_+}\frac{k^2+3k+4}{2}(z\bar{w})^k$ for all $ z,w\in\mathbb D.$ It follows from
Proposition $\ref{nms}$ that the sequence $\left \{\frac{2}{k^2+3k+4}\right\}_{k\in\mathbb Z_{+}}$ is not completely monotone.
Consequently, the multiplication operator $M_z$ on $\mathcal H(K_{1,1}+K_{3,1})$ is not subnormal.

For $\lambda>1,$ consider
the kernel $K_{{\lambda},1}+K_{3,1}.$
We claim that there exists a $\lambda_0>1$ such that $\left\{\frac{(1)_k}{(\lambda_0)_k+(3)_k}\right\}_{k\in\mathbb Z_{+}}$
is not completely monotone. If not, assume that it is a completely monotone sequence for all $\lambda>1$.
As $\lambda$ goes to $1,$ one may get that $\left \{\frac{2}{k^2+3k+4}\right\}_{k\in\mathbb Z_{+}}$
is completely monotone, which is a contradiction. Therefore,
we conclude that there exists a $\lambda_0>1$ such that the multiplication operator
$M_z$ on $\mathcal H(K_{\lambda_0,1}+K_{3,1})$ is not subnormal. By using properties of gamma function, one may verify that $K_{\lambda_0,1}\;$
and $K_{3,1}$ both satisfy $(a), (b)$ and $(c)$ of the Conjecture \ref{conj}.
This also provides a class of counterexamples for the Conjecture \ref{conj}.

\end{itemize}
\end{remark}

%%%%%%%%%%%%%%%%%%%%%%%%%%%%%%%%%%%%%%%%%%%%%%%%%%%%%%%%%%%%%%%%%%%%%%%%%%%%%%%%%%%%%%%%%%%%%%%%%%%%%%%%%%%%%%%%%%%%%%%%%%%%%%%%%%%%%%%%%%%
%%%%%%%%%%%%%%%%%%%%%%%%%%%%%%%%%%%%%%%%%%%%%%%%%%%%%%%%%%%%%%%%%%%%%%%%%%%%%%%%%%%%%%%%%%%%%%%%%%%%%%%%%%%%%%%%%%%%%%%%%%%%%%%%%%%%%%%%%%%

\begin{proposition}\label{power of comp alt}
Let $0<p\leq q\leq p+1$. Suppose $K_1(z,w)=\sum_{k \in \mathbb Z_+} a_k^p (z\bar{w})^k$ 
and $K_2(z,w)=\sum_{k \in \mathbb Z_+} a_k^q (z\bar{w})^k$ are any two reproducing kernels
such that $\{a_k\}_{k \in \mathbb Z_+}$ is a completely 
alternating sequence. Then the multiplication operator
$M_z$ on $\mathcal{H}(K_1+K_2)$ is subnormal.
\end{proposition}

\begin{proof}
Note that $$\frac{1}{a_k^p+a_k^q}=\frac{1}{a_k^p(1+a_k^{q-p})}.$$
Since $0\leq q-p\leq 1$ and $\{a_k\}_{k \in \mathbb Z_+}$ is completely alternating, it follows from \cite[Corollary 1]{At-Ra}
 that $\{a_k^{q-p}\}_{k \in \mathbb Z_+}$ is also completely alternating. 
 Thus so is $\{1+a_k^{q-p}\}_{k \in \mathbb Z_+}.$ Hence, by Remark \ref{rem2.6}, 
 $\{(1+a_k^{q-p})^{-1}\}_{k \in \mathbb Z_+}$ is completely monotone. 
By \cite[Corollary 4.1]{B-C-E}, $\{a_k^{-p}\}_{k \in \mathbb Z_+}$ is completely monotone.
  Now the proof follows as the product of two completely monotone sequences 
  is also completely monotone.    
\end{proof}

\begin{example}\label {ex2.13}
For any $p>0,$ let $K_p(z,w)$ be the positive definite kernel given by 
$$K_p(z,w):=\sum_{k \in \mathbb Z_+}(k+1)^p (z\bar{w})^k \qquad (z, w \in \mathbb D).$$
Then it is known that the multiplication operator $M_z$ on $\mathcal H(K_p)$ is subnormal
 with the representing measure $d\nu(x)=
\frac{(-\log x)^{p-1}}{\Gamma(p)} dx$ (cf. \cite[Theorem 4.3]{C-E}). By Proposition \ref{power of comp alt}, 
it follows that $M_z$ on $\mathcal H(K_p+K_q)$ is subnormal if $p\leq q\leq p+1.$

\end{example}

The next result also provides a class of counter-examples 
to the Conjecture \ref{conj}.
\begin{theorem} \label{thm2.16}
Consider the positive definite kernel $K_p$ given in Example \ref{ex2.13}. 
Then the multiplication operator $M_z$ on $\mathcal{H}(K_p+K_{p+2})$ is subnormal if and only if $p\geq 1.$
\end{theorem}

\begin{proof}
%Let $s=1.$ Then $\{\frac{1}{(k+1)((k+1)^2+1)}\}_{k \in \mathbb Z_+}$ is completely monotone 
%follows from
%$$2\int_0^1 x^k\sin^2 \left({\frac{-\log x}{2}}\right) dx=\frac{1}{(k+1)((k+1)^2+1)}.$$
For $x\in (0,1],$
let $g(x):=\frac{1}{\Gamma(p)}\int_0^{-\log x}(-\log x-y)^{p-1}\sin y ~dy$ 
and $d\nu(x)=g(x)dx.$ Then 
\beqn
\int_0^1 x^k d\nu(x) &=& \frac{1}{\Gamma(p)}\int_{y=0}^\infty\int_{x=0}^{e^{-y}}x^k(-\log x-y)^{p-1} dx\sin y ~dy
\cr &=& \frac{1}{\Gamma(p)}\int_{y=0}^\infty e^{-(k+1)y}\int_{u=0}^\infty
e^{-(k+1)u}u^{p-1} du \sin y ~dy
\cr &=&\frac{1}{\Gamma(p)}\int_{y=0}^\infty e^{-(k+1)y} 
\frac{\Gamma(p)}{(k+1)^p} \sin y ~dy
\cr &=&\frac{1}{(k+1)^p}\frac{1}{(k+1)^2+1}.
\eeqn
Thus the sequence $\{\frac{1}{(k+1)^p}\frac{1}{(k+1)^2+1}\}_{k \in \mathbb Z_+}$ is completely 
monotone if and only if the function $g(x)$ is non-negative a.e. Note that the function $g(x)$ 
is non-negative on $(0,1]$ a.e. if and only if the function $h(x):=g(e^{-x})$ is non-negative a.e. on $(0,\infty).$
Now 
\beqn 
h(x)=\frac{1}{\Gamma(p)} \int_0^{x}(x-y)^{p-1}\sin y ~dy =\frac{x^p}{\Gamma(p)}\int_0^1(1-y)^{p-1}sin(xy)~dy.
\eeqn
By \cite[Chapter 3, pp 439]{G-R}, we have 
$ h(x)= \frac{\sqrt x}{\Gamma(p)}s_{p-\tfrac{1}{2},\tfrac{1}{2}}(x),$ where 
$s_{p-\tfrac{1}{2},\tfrac{1}{2}}(x)$ is the Lommel's function of first kind. Thus, 
the sequence $\{\frac{1}{(k+1)^p}\frac{1}{(k+1)^2+1}\}_{k \in \mathbb Z_+}$ being completely monotone 
is equivalent to the non-negativity of the function $s_{p-\tfrac{1}{2},\tfrac{1}{2}}(x).$ 
If $p \geq 1$ then by \cite[Theorem A]{St}, we get that $s_{p-\tfrac{1}{2},\tfrac{1}{2}}(x) \geq 0$ for all $x > 0.$
The converse follows from  \cite[Theorem 2]{St}, which completes the proof.

\end{proof}

%%%%%%%%%%%%%%%%%%%%%%%%%%%%%%%%%%%%%%%%%%%%%%%%%%%%%%%%%%%%%%%%%%%%%%%%%%%%%%%%%%%%%%%%%%%%%%%%%%%%%%%%%%%%%%%%%%%5
%%%%%%%%%%%%%%%%%%%%%%%%%%%%%%%%%%%%%%%%%%%%%%%%%%%%%%%%%%%%%%%%%%%%%%%%%%%%%%%%%%%%%%%%%%%%%%%%%%%%%%%%%%%%%%%%%%%%%%%%%%%%%

\section{Multi-variable case}

Let $\mathbb Z_+^d$ denote the cartesian product ${\mathbb Z_+ \times \cdots \times \mathbb Z_+}~(d
~\text{times}).$ Let $\alpha =(\alpha_1, \cdots, \alpha_d) \in \mathbb Z_+^d,$
we write $|\alpha| :=\alpha_1 + \cdots + \alpha_d$ and $\alpha!=\alpha_1! \cdots \alpha_d!. $

If $T = (T_1, \cdots, T_d)$
is a $d$-tuple of commuting bounded linear operators $T_j~(1 \leq j
\leq d)$ on $\mathcal H$ then we set $T^*$ to be $(T^*_1,
\cdots, T^*_d)$ and $T^{\alpha}$ to be $T^{\alpha_1}_1\cdots T^{\alpha_d}_d.$

Given a commuting $d$-tuple $T$ of bounded
linear operators $T_1, \cdots, T_d$ on ${\mathcal H},$  set
\beqn Q_T(X) \mathrel{\mathop:}= \sum_{i=1}^dT^*_iXT_i \qquad (X \in B(\mathcal H)).\eeqn
For $X \in B(\mathcal H)$ and $ k \geq 1$ , one may define $Q^k_T(X):= Q_T(Q^{k-1}_T(X)),$ where $Q^0_T(X)=X.$

Recall that $T$ is said to be 
\begin{itemize}
 \item [(i)] {\it spherical contraction} if $ Q_T(I) \leq I.$
 \item [(ii)] {\it jointly left invertible} if there exists a positive
   number $c$ such that $ Q_T(I) \geq cI.$
\end{itemize}

For a jointly left invertible $T,$ the spherical Cauchy dual $T^{\mathfrak s}$ of $T$ is the $d$-tuple
$(T^{\mathfrak{s}}_1,T^{\mathfrak{s}}_2 \cdots, T^{\mathfrak{s}}_d),$ where
$T^{\mathfrak{s}}_i:=T_i (Q_T(I))^{-1}~(i=1, 2, \cdots, d). $
We say that $T$ is a {\it joint complete hyperexpansion} if
\beqn B_n(T) \mathrel{\mathop:}= \sum_{k =0}^n (-1)^{k} {n
\choose k} Q^{k}_T(I)\leq 0 \qquad (n\geq 1).\eeqn

Throughout this section $\mathbb B$ denotes the open unit ball
$\{z\in \mathbb C^d: |z_1|^2+\cdots + |z_d|^2<1\}$ and $\partial \mathbb B$ denotes the unit sphere
$\{z \in \mathbb C^d:|z_1|^2+\cdots + |z_d|^2=1 \}$ in $\mathbb C^d.$

Let $\{\beta_{\alpha}\}_{\alpha \in \mathbb Z^d_+}$ be a multi-sequence of
positive numbers.
Consider the Hilbert space $H^2(\beta)$ of formal power series
$f(z)=\sum_{\alpha \in  \mathbb Z^d_+}\hat{f}(\alpha) z^{\alpha} $ such that
$$ \|f\|^2_{H^2(\beta)}=\sum_{\alpha \in  \mathbb Z^d_+}|\hat{f}(\alpha)|^2 \beta^2_{\alpha} <
\infty.$$

The Hilbert space $H^2(\beta)$ is said to be {\it spherically balanced} if the norm on $H^2(\beta)$
admits the slice representation $[\nu, H^2(\gamma)]$, that is, there exist a {\it Reinhardt measure} $\nu$
 and a Hilbert
space $H^2(\gamma)$ of formal power series in one variable such that
\beqn \|{f}\|^2_{H^2(\beta)} = \int_{\partial \mathbb
 B}\|{f_z}\|^2_{H^2(\gamma)}d\nu(z) \qquad (f \in
 H^2(\beta)), \eeqn where $\gamma =\{\gamma_k\}_{k \in \mathbb Z_+}$ is given by the relation
$\beta_{\alpha}=\gamma_{|\alpha|}\|z^{\alpha}\|_{L^2(\partial \mathbb B, \nu)}$ for all $\alpha \in \mathbb Z_+^d.$
Here, by the Reinhardt measure, we mean a $\mathbb T^d$-invariant
finite positive Borel measure supported in $\partial \mathbb B,$
where $\mathbb T^d$ denotes the the unit $d$-torus
$\{z \in \mathbb C^d:|z_1|=1,\cdots,|z_d|=1\}.$
For more details on spherically balanced Hilbert spaces, we refer to \cite{Ch-K}.

The following lemma has been already recorded in \cite[Lemma 4.3]{Ch-K}. We include a statement 
for ready reference.

\begin{lemma}\label{lem3.1}
Let $H^2(\beta)$ be a spherically balanced Hilbert space and let $[\nu, H^2(\gamma)]$
be the slice representation for the norm on $H^2(\beta).$
Consider the $d$-tuple $M_z=(M_{z_1}, \cdots, M_{z_d})$ of multiplication by the
co-ordinate functions $z_1, \cdots, z_d$ on $H^2(\beta).$ 
Then for every $n \in \mathbb Z_+$ and $\alpha \in \mathbb Z^d_+,$ \beqn
\inp{B_n(M_z)z^{\alpha}}{z^{\alpha}}=
\sum_{k=0}^n (-1)^k {n \choose k} \inp{Q^k_{M_z}(I)z^{\alpha}}{z^{\alpha}} =
\sum_{k=0}^n (-1)^k {n \choose k} \gamma^2_{k+|\alpha|}
\|z^{\alpha}\|^2_{L^2(\partial \mathbb B, \nu)}. \eeqn 
\end{lemma}

If the interior of the point spectrum $\sigma_p(M^*_z)$ of $M_z^*$ is non-empty then $H^2(\beta)$
may be realized as a reproducing kernel Hilbert space $\mathcal H(K)$
\cite[Propositions 19 and 20]{J-L}, where the reproducing kernel $K$ is given by
\beqn K(z,w)=\sum_{\alpha \in  \mathbb Z^d_+}\frac{ z^{\alpha} \bar{w}^{\alpha}}{\beta^2_{\alpha}}
\qquad (z, w \in \sigma_p(M^*_z)).\eeqn
This has lead to the following definition.
\begin{definition}
 Let $\mathcal H(K)$ be a reproducing kernel Hilbert space
defined on the open unit ball $\mathbb B$
with reproducing kernel
$K(z,w)=\sum_{\alpha \in  \mathbb Z^d_+}a_{\alpha} z^{\alpha} \bar{w}^{\alpha}$
for all $z, w \in \mathbb B.$
We say that $K$ is {\it balanced kernel} if
$\mathcal H(K)$ is a spherically balanced Hilbert space.
Further, the multiplication $d$-tuple $M_z$ on $\mathcal H(K)$
may be called as {\it balanced multiplication tuple}.

\end{definition}

\begin{remark} \label{rem1.2}
The spherical Cauchy dual  $M_z^{\mathfrak s}$ of a jointly
left invertible balanced multiplication tuple $M_z$ can be seen as a multiplication
$d$-tuple $M_z^{\mathfrak s}=(M_{z_1}^{\mathfrak s},
\cdots, M_{z_d}^{\mathfrak s})$ of multiplication by the co-ordinate functions
$z_1, \cdots, z_d$ on  $H^2(\beta^{\mathfrak s}),$ where
\beqn \beta_{\alpha}^{\mathfrak s}=\frac{1}{\gamma_{|\alpha|}}\|z^{\alpha}\|
_{L^2(\partial \mathbb B, \nu)} \qquad (\alpha \in \mathbb Z_+^d).\eeqn
In other words, the norm on $H^2(\beta^{\mathfrak s})$
admits the slice representation $[\nu, H^2(\gamma^{\prime})],$ where
$\gamma^{\prime}_k=1/\gamma_k$ for all $k \in \mathbb Z_+. $
\end{remark}

\begin{proposition}\label{prop1.3}
If $ K_1(z,w)=\sum_{\alpha \in  \mathbb Z^d_+}a_{\alpha} z^{\alpha} \bar{w}^{\alpha}$ and
$ K_2(z,w)=\sum_{\alpha \in  \mathbb Z^d_+}b_{\alpha} z^{\alpha} \bar{w}^{\alpha}$ are any two balanced kernels
 with the slice representations $[\nu, H^2(\gamma_1)]$
and $[\nu, H^2(\gamma_2)]$ respectively.
Then $K_1+K_2$ is a balanced kernel with the slice representation $[\nu/2, H^2(\gamma)],$ where $\gamma=\{\gamma_k\}$
is given by the relation \beqn \gamma_k=\frac{\sqrt{2}\gamma_{k,1}\gamma_{k,2}}{(\gamma^2_{k,1}+\gamma^2_{k,2})^{1/2}} \qquad (k \in \mathbb Z_+).\eeqn
\end{proposition}

\begin{proof}
For every $\alpha \in \mathbb Z_+^d,$ we have
\beqn a_{\alpha}+b_{\alpha} = \frac{1} {\gamma^2_{|\alpha|, 1}\|z^{\alpha}\|^2_{L^2(\partial \mathbb B, \nu)}}+
\frac{1}{\gamma^2_{|\alpha|, 2}\|z^{\alpha}\|_{L^2(\partial \mathbb B, \nu)}}
= \frac{(\gamma^2_{|\alpha|, 1}+\gamma^2_{|\alpha|, 2})}{\gamma^2_{|\alpha|, 1}\gamma^2_{|\alpha|, 2}
\|z^{\alpha}\|^2_{L^2(\partial \mathbb B, \nu)}}. \eeqn
Therefore
% the norm of the monomials  $\{z^{\alpha}\}_{\alpha \in \mathbb Z_+^d}$  in $\mathcal H(K_1+K_2)$ is
\beqn \|z^{\alpha}\|^2_{\mathcal H(K_1+K_2)} = \frac{2\gamma^2_{|\alpha|, 1}\gamma^2_{|\alpha|, 2}}
{(\gamma^2_{|\alpha|, 1}+\gamma^2_{|\alpha|, 2})}\|z^{\alpha}\|^2_{L^2(\partial \mathbb B, \nu/2)}
=\gamma^2_{|\alpha|}\|z^{\alpha}\|^2_{L^2(\partial \mathbb B, \nu/2)} \eeqn
for all $\alpha \in \mathbb Z_+^d.$
Since $\{z^{\alpha}\}_{\alpha \in \mathbb Z_+^d}$ forms an orthogonal subset of $L^2(\partial \mathbb B, \nu/2),$
the conclusion follows immediately.
\end{proof}

\begin{remark} The conclusion of the Proposition \ref{prop1.3} still holds even if we chose two different Reinhardt measures $\nu_1$ and $\nu_2$
in the slice representations of $K_1$ and $K_2,$  such that for some sequence of positive real numbers  $\{h_k\}_{k \in \mathbb Z_+},$ 
$\|z^{\alpha}\|_{L^2(\partial \mathbb B, \nu_1)}=h_{|\alpha|}\|z^{\alpha}\|_{L^2(\partial \mathbb B, \nu_2)}$ for all $\alpha \in \mathbb Z_+^d.$  
For every $j=1,2,$ it is easy to verify that \beqn \sum_{i=1}^d
\frac{\|z^{\alpha+\varepsilon_i}\|^2_{L^2(\partial \mathbb B,
\nu_j)}}{\|z^{\alpha}\|^2_{L^2(\partial \mathbb B, \nu_j)}}=1. \eeqn
This implies that  $\{h_k\}_{k \in \mathbb Z_+}$ is a constant
sequence, say $c$. Now, a routine argument, using the Stone-Weierstrass
theorem, we conclude that  $\mu_1=c^2\mu_2.$

\end{remark}

Let $\mathcal K_{\nu}$ denote the class of all balanced kernels with the
following properties:
\begin{itemize}
\item[(i)]For all $K \in \mathcal K_{\nu},$ the norm on $\mathcal H(K)$ admits the slice representations
 with fixed Reinhardt measure $\nu.$
 \item[(ii)]For every member $K$ of $\mathcal K_{\nu},$ the multiplication operator $M_z$ defined on $\mathcal H(K)$ is jointly left invertible.
\item[(iii)] The Cauchy dual tuple $M_z^{\mathfrak s}$ of is a joint complete hyperexpansion.

\end{itemize}

\begin{lemma}\label{lem1.5}
For every member $K$ of $\mathcal K_{\nu},$ the multiplication operator tuple $M_z$ defined on
$\mathcal H(K)$ is a subnormal spherical contraction.
 \end{lemma}

\begin{proof}

Let $K \in \mathcal K_{\nu} $ and $[\nu, H^2(\gamma)]$ be the slice representation for the norm on $\mathcal H(K).$
Note that the Cauchy dual $M_z^{\mathfrak s}$ of $M_z$ is a balanced multiplication tuple with slice representation
$[\nu, H^2(1/\gamma)]$ (see Remark \ref{rem1.2}).
Since $M_z^{\mathfrak s}$ is a joint complete hyperexpansion.
It follows from Lemma \ref{lem3.1} that $\{1/\gamma^2_k\}_{k \in \mathbb Z_+}$ is a
completely alternating sequence. Therefore, by Remark \ref{rem2.6},  $\{\gamma^2_k\}_{k \in \mathbb Z_+}$ is
completely monotone sequence. Now again by applying Lemma \ref{lem3.1}, we conclude that
the multiplication operator $M_z$ is a subnormal spherical contraction.
\end{proof}

\begin{theorem}\label{thm1.4}
If $K_1$ and $K_2$ are any two members of $\mathcal K_{\nu}$ then 
 the multiplication operator $M_z$ on $\mathcal H(K_1+K_2)$ is a subnormal spherical contraction.
\end{theorem}

\begin{proof}
Note that the norm on $\mathcal H(K_1+K_2)$ admits the slice representation
$[\nu/2, H^2(\gamma)],$ where
$ \gamma^2_k=2\frac{\gamma^2_{k,1}\gamma^2_{k,2}}{\gamma^2_{k,1}+\gamma^2_{k,2}}$ for all $k \in \mathbb Z_+$
(see Proposition \ref{prop1.3}). It follows from the proof of Lemma \ref{lem1.5} that  $\{1/\gamma^2_{k,1}\}_{k \in \mathbb Z_+}$  and
$\{1/\gamma^2_{k,2}\}_{k \in \mathbb Z_+}$ are completely alternating. So their sum, that is,
$\{1/\gamma^2_k\}_{k \in \mathbb Z_+}$ is a completely alternating sequence.
Now the conclusion follows by imitating the argument given in
Lemma \ref{lem1.5}.
\end{proof}

For $\lambda >0,$ consider the positive definite kernel
$K_{\lambda}$ given by \beqn
K_{\lambda}(z,w)=\frac{1}{(1-\inp{z}{w})^{\lambda}} \qquad (z, w \in
\mathbb B). \eeqn The norm on $H(K_{\lambda})$ admits the slice
representation $[\sigma, H^2(\gamma)],$ where $\sigma$ denotes the
normalized surface area measure on $\partial \mathbb B$ and
$\gamma^2_k= \frac{(d)_k}{(\lambda)_k}$ for all $k \in \mathbb Z_+.$
It is well known that the multiplication operator $M_{z, \lambda}$
on $H(K_{\lambda})$ is a subnormal contraction if and only if
$\lambda \geq d.$ The same can also be verified by using Lemma \ref{lem3.1}  
and part $(i)$ of Theorem \ref{thm2.8}. Similarly, by
using Lemma \ref{lem3.1} and part $(ii)$ of Theorem
\ref{thm2.8}, one may conclude that the Cauchy dual tuple $M_{z,
\lambda}^{\mathfrak s}$ is a joint complete hyperexpansion if and
only if $d \leq \lambda \leq d+1.$
Thus, if we choose $\lambda$ and $\lambda^{\prime}$ are such that $d \leq \lambda,
\lambda^{\prime} \leq d+1.$ Then $K_{\lambda}$ and
$K_{\lambda^{\prime}} \in \mathcal K_{\sigma.}$ It now follows from
Theorem \ref{thm1.4} that the multiplication operator $M_z$ on
$\mathcal H (K_{\lambda}+K_{\lambda^{\prime}})$ is subnormal. This
is also included in the following example.

\begin{example}
Let $0<d \leq \lambda'\leq \lambda\leq \lambda'+1.$ Note that the
norm on $\mathcal H (K_{\lambda}+K_{\lambda^{\prime}})$ admits the
slice representation $[\sigma /2, H^2(\gamma)],$ where $\gamma^2_k=
\frac{2(d)_k}{(\lambda)_k+(\lambda^{\prime})_k}$ for all $k \in
\mathbb Z_+.$ From the proof of Proposition \ref{prop2.9}, it is
clear that $\{\gamma^2_k\}_{k \in \mathbb Z_+}$ is completely
monotone. Hence, the multiplication operator $M_z$ on $\mathcal H
(K_{\lambda}+K_{\lambda^{\prime}})$ is subnormal.
\end{example}

A $d$-tuple $S=(S_1, \cdots, S_d)$ of commuting bounded linear
operators $S_1, \cdots, S_d $ in ${\mathcal B}({\mathcal H})$ is a {\it spherical
isometry} if $S^*_1S_1 + \cdots + S^*_dS_d=I.$ In other words, $ Q_S(I) = I.$ 
The most interesting example of a spherical isometry is the Szeg\"o
$d$-shift; that is, the $d$-tuple $M_z$ of multiplication operators
$M_{z_1}, \cdots, M_{z_d}$ on the Hardy space $H^2(\partial \mathbb
B)$ of the unit ball.
% Let $\mathbb C[z]=\mathbb C[z_1, \cdots, z_m]$ denote the
% complex vector space of polynomials in the variables
% $z_1,\cdots,z_m.$
% For a finite positive Borel measure $\nu$ supported on the unit sphere $\partial
% \mathbb B,$ let $P^2(\partial \mathbb B,\nu)$ be the closure of
% ${\mathbb C[z]}|_{\partial \mathbb B}$ in ${L^2(\partial \mathbb
% B,\nu)}.$ 
% 

% Recall that, every spherical isometry weighted multi-shifts is unitarily equivalent to 

Let $\nu $ be a Reinhardt measure. 
Consider the multiplication $d$-tuple $M_z$ on a reproducing kernel Hilbert space
 $\mathcal H(K^{\nu})$ determined by the reproducing kernel 
 \beq \label{eq3.2} K^{\nu}(z,w)=\sum_{\alpha \in \mathbb Z_+^d} \frac{z^{\alpha}
\bar{w}^{\alpha}}{\|z^{\alpha}\|^2_{L^2(\partial \mathbb B, \nu)}}\qquad (z, w \in \mathbb B). \eeq
Note that $M_z$ is a spherical isometry.
%  (cf. \cite[Theorem 2.1]{Cu-Sa}).
In this case, the norm on $\mathcal H(K^{\nu})$ admits
the slice representation $[\nu,H^2(\mathbb D)],$
where $H^2(\mathbb D)$ is the Hardy space of the unit disc.
% \begin{theorem} \cite[Theorem 2.1]{Cu-Sa}
% Every cyclic spherical isometry $m$-tuple is unitarily equivalent to the multiplication
% $m$-tuple $M_z$ on $P^2(\partial \mathbb B, \nu)$ for some finite
% positive Borel measure $\nu$ supported on the unit sphere $\partial
% \mathbb B.$
% \end{theorem
\begin{theorem}
Let $K^{\nu}$ be the reproducing kernel given as in equation \eqref{eq3.2} and 
$\tilde{K}$ be any balanced kernel with the slice representation $[\nu, H^2(\tilde{\gamma)}].$
Assume that the multiplication operator $M_z$ 
on $\mathcal{H}(K^{\nu}+\tilde{K})$ is subnormal. Then the multiplication operator on 
$\mathcal H(\tilde{K})$ is subnormal.
\end{theorem}

\begin{proof}
Observe that the norm on $\mathcal{H}(K^{\nu}+\tilde{K})$ admits the slice representation $[\nu/2, H^2(\gamma)],$
where ${\gamma^2_k}= 2 (1+1/{\tilde{\gamma}^2_k})^{-1}$ for all $k \in \mathbb Z_+.$
Since $M_z$ 
on $\mathcal{H}(K^{\nu}+\tilde{K})$ is subnormal, it follows from Lemma \ref{lem3.1} that $\{\gamma^2_k\}_{k \in \mathbb Z_+}$
is a completely monotone sequence. Hence, $\{(1+1/{\tilde{\gamma}^2_k)^{-1}}\}_{k \in \mathbb Z_+}$ is a completely monotone sequence.
If we replace $a_k$ by $1/{\tilde{\gamma}^2_k}$ in the proof of the Proposition \ref{pro2.9}, we get that $\{\tilde{\gamma}^2_k\}_{k \in \mathbb Z_+}$
is completely monotone. Now, by applying Lemma \ref{lem3.1}, we conclude that the multiplication operator on 
$\mathcal H(\tilde{K})$ is subnormal.
 
\end{proof}

We conclude the paper with the following questions:
\begin{question}
In view of Proposition \ref{prop2.9} and Theorem \ref{thm2.16}, it is natural to ask that
 \begin{itemize}
  \item [(i)]what is the necessary and sufficient condition for the multiplication operator $M_z$ on $\mathcal{H}(K_{\lambda,\mu}+K_{\lambda^{\prime},\mu})$ to be subnormal?
  \item [(ii)]what is the necessary and sufficient condition for the multiplication operator $M_z$ on $\mathcal{H}(K_p+K_{q})$ to be subnormal?
  \end{itemize}
\end{question}

\begin{question}
Let $K^{\nu}$ be the reproducing kernel given as in equation \eqref{eq3.2} and  $\tilde{K}$
be any positive definite kernel given by
\beqn \tilde{K}(z,w):=\sum_{\alpha \in  \mathbb Z^d_+}a_{\alpha} z^{\alpha} \bar{w}^{\alpha} \qquad (z, w \in \mathbb B). \eeqn 
Assume that the $d$-tuple $M_z=(M_{z_1}, \cdots, M_{z_d})$ of multiplication by the
co-ordinate functions $z_1, \cdots, z_d$ on $\mathcal{H}(K^{\nu}+\tilde{K})$ is subnormal.
Is it necessary that the multiplication operator on $\mathcal H(\tilde{K})$ subnormal?
 
\end{question}

\medskip \textit{Acknowledgments}.
% We express our sincere thanks to Professor Gadadhar Misra and
% Professor Sameer Chavan for several fruitful suggestions, which
% resulted in a considerable refinement of the original draft.
We express our sincere thanks to Prof. G. Misra for many fruitful conversations and
suggestions in the preparation of this paper. We would also like to thank Prof. S. Chavan
for his many useful comments and careful reading of the manuscript.

\end{document}